\documentclass{amsart}
\usepackage[english]{babel}
\usepackage{amsmath}
\usepackage{amsthm}
\usepackage{amssymb}
\usepackage[all]{xy}
\usepackage{latexsym}
\usepackage{exscale}
\usepackage{comment}

\usepackage{amssymb}
\usepackage{amscd}

\usepackage[all]{xy}

\usepackage{amsbsy}

\def\frak{\mathfrak}
\def\Bbb{\mathbb}
\def\Cal{\mathcal}

\newtheorem{prop}[subsection]{Proposition}
\newtheorem*{prop*}{Proposition}
\newtheorem{thm}[subsection]{Theorem}
\newtheorem*{thm*}{Theorem}
\newtheorem{lem}[subsection]{Lemma}
\newtheorem*{lem*}{Lemma}

\newtheorem*{kor*}{Corollary}
\newtheorem{rem}[subsection]{Remark}
\newtheorem{df}[subsection]{Definition}
\newtheorem{ex}[subsection]{Example}

\def\gop{\frak{p}}

\newcommand{\frg}{\mathfrak{g}}
\newcommand{\frh}{\mathfrak{h}}

\newcommand{\frl}{\mathfrak{l}}

\newcommand{\frp}{\mathfrak{p}}
\newcommand{\frqqq}{\mathfrak{q}}

\newcommand{\frsl}{\mathfrak{sl}}

\newcommand{\caS}{{\mathcal S}}

\def\bbC{\mathbb{C}}
\def\bbP{\mathbb{P}}

\let\ccdot\cdot
\def\cdot{\hbox to 2.5pt{\hss$\ccdot$\hss}}

\newcommand{\gog}{{\mathfrak g}}
\newcommand{\cO}{{\Cal O}}
\newcommand{\bV}{{\Bbb V}}

\newcommand{\bC}{{\Bbb C}}

\newcommand{\eq}{\begin{equation}}
\newcommand{\eeq}{\end{equation}}
\newcommand{\eqn}{\begin{equation*}}
\newcommand{\eeqn}{\end{equation*}}

\newcommand{\la}{\lambda}

\renewcommand{\phi}{\varphi}

\let\ssize\scriptstyle
\newif\ifFIRST\newdimen\MAXright\MAXright0pt
\def\sdynkin{\bgroup\eightpoint\dynkin}
\def\endsdynkin{\enddynkin\egroup}
\def\dynkin{\bgroup\FIRSTtrue\hskip.5em\setbox1\hbox{$\diagup$}%
	\setbox2\hbox{$\diagdown$}%
	\setbox0\hbox to2\wd1{\hrulefill}%
	\setbox3\hbox{$\bullet$}%
	\setbox4\hbox{$\times$}%
	\setbox7\hbox{$\circ$}
	\def\whiteroot##1{\ifFIRST\setbox5\hbox{$##1$}\ifdim\wd5>1.3em
		\hskip-.5em\hskip.5\wd5\fi\fi\FIRSTfalse
		\hskip-.25em\raise1.5\wd3\hbox to0pt{\hss\hskip.45em$
			\ssize##1$\hss}\copy7\hskip-.25em\setbox6\hbox{$##1$}
		\MAXright\wd6}
	\def\root##1{\ifFIRST\setbox5\hbox{$##1$}\ifdim\wd5>1.3em%
		\hskip-.5em\hskip.5\wd5\fi\fi\FIRSTfalse%
		\hskip-.25em\raise1.5\wd3\hbox to0pt{\hss\hskip.45em$%
			\ssize##1$\hss}\copy3\hskip-.25em\setbox6\hbox{$##1$}%
		\MAXright\wd6}%
	\def\whitedroot##1{\ifFIRST\setbox5\hbox{$##1$}\ifdim\wd5>1.3em
		\hskip-.5em\hskip.5\wd5\fi\fi\FIRSTfalse
		\hskip-.25em\lower1.8\wd3\hbox to0pt{\hss\hskip.45em$
			\ssize##1$\hss}\copy7\hskip-.25em\setbox6\hbox{$##1$}
		\MAXright\wd6}%
	\def\whiterroot##1{\hskip-.25em\copy7\hbox to0pt{\hskip.3em$\ssize##1$\hss}%
		\hskip-.25em\setbox6\hbox{\hskip.6em$##1##1$}%
		\MAXright\wd6}%
	\def\droot##1{\ifFIRST\setbox5\hbox{$##1$}\ifdim\wd5>1.3em%
		\hskip-.5em\hskip.5\wd5\fi\fi\FIRSTfalse%
		\hskip-.25em\lower1.8\wd3\hbox to0pt{\hss\hskip.45em$%
			\ssize##1$\hss}\copy3\hskip-.25em\setbox6\hbox{$##1$}%
		\MAXright\wd6}%
	\def\rroot##1{\hskip-.25em\copy3\hbox to0pt{\hskip.3em$\ssize##1$\hss}%
		\hskip-.25em\setbox6\hbox{\hskip.6em$##1##1$}%
		\MAXright\wd6}%
	\def\norroot##1{\hskip-.36em\copy4\hbox to0pt{\hskip.3em$\ssize##1$\hss}%
		\hskip-.48em\setbox6\hbox{\hskip.6em$##1##1$}%
		\MAXright\wd6}%
	\def\noroot##1{\ifFIRST\setbox5\hbox{$##1$}\ifdim\wd5>1.3em%
		\hskip-.5em\hskip.5\wd5\fi\fi\FIRSTfalse%
		\hskip-.36em\raise1.5\wd3\hbox to0pt{\hss\hskip.6em$%
			\ssize##1$\hss}\copy4\hskip-.38em\setbox6\hbox{$##1$}%
		\MAXright\wd6}%
	\def\nodroot##1{\ifFIRST\setbox5\hbox{$##1$}\ifdim\wd5>1.3em%
		\hskip-.5em\hskip.5\wd5\fi\fi\FIRSTfalse%
		\hskip-.36em\lower1.8\wd3\hbox to0pt{\hss\hskip.6em$%
			\ssize##1$\hss}\copy4\hskip-.38em\setbox6\hbox{$##1$}%
		\MAXright\wd6}%
	\def\nolink{\hskip\wd0}
	\def\link{\raise.22em\copy0}%
	\def\llink##1{\raise.32em\copy0\hskip-\wd0%
		\raise.12em\copy0\hskip-.5\wd0\hbox to0pt{\hss$##1$\hss}\hskip.5\wd0}%
	\def\lllink##1{\raise.22em\copy0\hskip-\wd0\raise.32em\copy0\hskip-\wd0%
		\raise.12em\copy0\hskip-.5\wd0\hbox to0pt{\hss$##1$\hss}\hskip.5\wd0}%
	\def\rootupright##1{\hbox to0pt{\raise.45em\copy1\hskip-.25em\raise1.3\ht1%
			\hbox{\copy3\hskip.3em$\ssize##1$}\hss}%
		\setbox6\hbox{\hskip.6em\copy1\copy1$##1##1$}%
		\ifdim\MAXright<\wd6\MAXright\wd6\fi}%
	\def\whiterootupright##1{\hbox to0pt{\raise.45em\copy1\hskip-.25em\raise1.3\ht1
			\hbox{\copy7\hskip.3em$\ssize##1$}\hss}
		\setbox6\hbox{\hskip.6em\copy1\copy1$##1##1$}
		\ifdim\MAXright<\wd6\MAXright\wd6\fi}
	\def\norootupright##1{\hbox to0pt{\raise.45em\copy1\hskip-.36em\raise1.3\ht1%
			\hbox{\copy4\hskip.3em$\ssize##1$}\hss}%
		\setbox6\hbox{\hskip.6em\copy1\copy1$##1##1$}%
		\ifdim\MAXright<\wd6\MAXright\wd6\fi}%
	\def\rootdownright##1{\hbox to0pt{\raise-.5em\copy2\hskip-.25em\raise-1.35\ht1%
			\hbox{\copy3\hskip.3em$\ssize##1$}\hss}\setbox6%
		\hbox{\hskip.6em\copy2\copy2$##1##1$}%
		\ifdim\MAXright<\wd6\MAXright\wd6\fi}%
	\def\whiterootdownright##1{\hbox to0pt{\raise-.5em\copy2\hskip-.25em\raise-1.35\ht1
			\hbox{\copy7\hskip.3em$\ssize##1$}\hss}\setbox6
		\hbox{\hskip.6em\copy2\copy2$##1##1$}
		\ifdim\MAXright<\wd6\MAXright\wd6\fi}
	\def\rootdown##1{\hbox to0pt{\hskip-.05em\vrule height.25em depth.65em%
			\hskip-.25em\raise-.95em\hbox{\copy3\hskip.3em$\ssize##1$}\hss}%
		\setbox6\hbox{$##1$}%
		\ifdim\MAXright<\wd6\MAXright\wd6\fi}%
	\def\whiterootdown##1{\hbox to0pt{\hskip-.05em\vrule height.25em depth.65em
			\hskip-.25em\raise-.95em\hbox{\copy7\hskip.3em$\ssize##1$}\hss}
		\setbox6\hbox{$##1$}
		\ifdim\MAXright<\wd6\MAXright\wd6\fi}
	\def\dots{\hskip.5em\cdots\hskip.5em}}%
\def\enddynkin{\ifdim\MAXright>1em\hskip.5\MAXright\else\hskip.5em\fi\egroup}%
 
\begin{document} 

\title[]{BGG complexes in singular infinitesimal character for type A}
\author{Pavle Pand\v zi\'c}
\address[Pand\v zi\'c]{Department of Mathematics, University of Zagreb, Bijeni\v cka 30, 10000 Zagreb, Croatia}
\email{pandzic@math.hr}
\thanks{The first named author is
supported by grant no. 4176 of the Croatian Science Foundation, and by the Center of Excellence QuantiXLie.} 
\author{Vladim\'{\i}r Sou\v cek}
\address[Sou\v cek]{Matematick\'y \'ustav UK, Sokolovsk\'a 83, 186 75 Praha 8, Czech Republic}
\email{soucek@karlin.mff.cuni.cz}
\thanks{The second named author is
supported by the grant GA CR P201/12/G028}
\date{}
\begin{abstract} We give a geometric construction of the BGG resolutions in singular infinitesimal character in the case of 1-graded complex Lie algebras of type A.
\end{abstract}

\maketitle
\section{Introduction}
\label{sec:intro}

The BGG complexes were first constructed by Bernstein, Gelfand and Gelfand in \cite{BGG}. Let $\gog$ be a simple complex Lie algebra. For each finite dimensional $\gog$-module $\bV,$  they constructed a resolution of  $\bV$ by direct sums of Verma modules. It was soon generalized by Lepowsky from the Borel case to the case of general parabolic subalgebras and generalized Verma modules \cite{L}.

It is well known that homomorphisms between (generalized) Verma modules correspond dually to invariant differential operators acting between the spaces of sections of the corresponding homogeneous vector bundles.
A geometric construction of the dual BGG resolutions over the flag manifolds (and their extensions to the ``curved" Cartan geometries) was described in \cite{CSS,CD}.   A typical feature of these resolutions is the fact that the whole orbit of the affine action of the Weyl group $W$ of $\gog$ is involved in the construction of these resolutions and all integral affine orbits in regular infinitesimal character are covered. The more complicated cases can be obtained from the basic one (corresponding to the trivial representation of $\gog$)
by the Zuckerman translation.

The aim of this paper is to give a geometric construction of BGG resolutions for integral orbits of the affine action of $W$ in singular infinitesimal character on Grassmannian manifolds $G(k,n) $ in the dual language of invariant differential operators on the big cell in  $G(k,n).$ 
This is closely related 
to the problem of classification of
Kostant modules in singular infinitesimal character, which was solved in the
paper by Boe and Hunziker \cite{BH}. There are two differences in the approach.

First,  we concentrate our attention only on the Kostant modules which are
resolved by complete BGG resolutions. The results about classification of Kostant
modules in \cite{BH} are based on the Enright-Shelton equivalence of categories between
singular and regular infinitesimal character \cite{ES}.  
The image of the finite-dimensional representation under the Enright-Shelton
correspondence is an infinite dimensional simple module, which is resolved
by the complete BGG resolution covering the whole orbit.
The Enright-Shelton theory applies to the cases when the parabolic subalgebra $\gop$ corresponds to the 1-graded Lie algebra.
This describes the BGG resolutions on the algebraic side of the story.

We are working in the dual language of invariant differential operators.
We classify and construct explicitly the set of all BGG resolutions for all
orbits in singular infinitesimal character (for all 1-graded cases in type A).
Our methods do not use the Enright-Shelton results.

Second, the results in \cite{ES} describe the Hasse diagram in singular infinitesimal character
by the corresponding Hasse diagram in regular infinitesimal character. The result does not
depend on a specific choice of the singularity (described by a subset $S$ of the set of simple roots), but it depends only on the size of the set. However, even if the overall structure    of the Hasse diagram is the same,  we shall show that
the type of intertwining differential operators in the BGG resolution  (and their order)
depend on a choice of $S.$  We shall give an explicit construction of differential operators in the complex and we show how their properties depend on the choice of $S.$

To describe our results more precisely, let $G=SL(n,\bbC)$, let $k$ be an integer between 1 and $n/2$, and let $P$ be the following maximal parabolic subgroup of 
$G$: 
\[
P=\left\{\begin{pmatrix} k\times k & k\times (n-k) \cr
0          & (n-k)\times (n-k)
\end{pmatrix}\right\}.
\]
The Levi subgroup of $P$ is $L=S(GL(k,\bbC)\times GL(n-k,\bbC))$. The corresponding Lie algebras are denoted by $\gop\subset\gog.$

Let $\frh\subset\frg$ be the Cartan subalgebra consisting of diagonal matrices. We will work with weights $\mu\in\frh^*\cong\bbC^n/\bbC(1,\dots,1)$, and for notational convenience, we will not require 
$\mu=(\mu_1,\dots,\mu_n)$ to satisfy the condition $\sum_i\mu_i=0$. We take the standard choice of positive roots, so that $\mu$ is dominant for $\frg$ if and only if $\mu_1\geq\dots\geq\mu_n$, and $\mu$ is dominant regular for $\frg$ if all the inequalities are strict. To fix a choice of a dominant $\mu$, we can require $\mu_n=0$. The weight $\mu$ is integral for $\gog,$
if all differences $\mu_i-\mu_{i+1}$ are integers.

We can view elements of $\frh^*$ also as weights for $\frl$; in this case, we separate the first $k$ coordinates from the last $n-k$ coordinates by a bar. A weight $\nu$ is regular and dominant for $\frg$ (i.e., for $\frl$) if and only if 
$\nu=(\nu_1,\ldots\nu_k|\nu_{k+1},\ldots,\nu_n)$
with $\nu_1>\nu_2>\ldots>\nu_k$ and $\nu_{k+1}>\nu_{k+2}>\ldots >\nu_n.$

Let $\mu\in\frh^*$ be a parameter which is singular and integral for $\frg$ but which has Weyl group conjugates that are regular for $\frl$. This means that $\mu$ has some coordinates that are repeated, but each of them only twice, so that a Weyl group element can put one member of each repeated pair among the first $k$ coordinates and the other among the last $n-k$ coordinates. Let $l$ be the number of repeated pairs; it follows that $1\leq l\leq k$. We can record the singularity precisely by conjugating $\mu$ to the $\frg$-dominant chamber, and then defining the singularity set as
\eq
\label{sing set}
S=\{s_1,\dots,s_l\},\quad s_1<\ldots<s_l,
\eeq
if the repeated pairs of coordinates are 
\[
\mu_{s_1}=\mu_{s_1+1},\dots,\mu_{s_l}=\mu_{s_l+1}.
\]
Note that necessarily $s_{r+1}\geq s_r+2$, for $r=1,\dots,l-1$.

  We need to understand the structure
of the part $Orb$ of the $W$-orbit of a fixed $\mu$ as above  consisting of regular $\gop$-dominant weights, because it is the set of inducing weights for singular BGG resolutions in one fixed infinitesimal block. 
The structure of $Orb$ depends, of course, on the order of singularity $l$.
We can describe it as follows.
Consider the Lie algebra $\gog'=\frsl(n-2l,\bC)$ and its Weyl
group $W'.$ Let $\frp'$ be the parabolic subalgebra of $\frg'$ consisting of block upper triangular matrices with diagonal blocks of sizes $k-l$ and $n-k-l$. Let $\frl'$ be the Levi subalgebra of $\frp'$ consisting of the block diagonal matrices in $\frp'$. 

There is a simple one-to-one correspondence between the set $Orb$ for $\mu$ and the analogous set $Orb'$ for $\mu'$, i.e., the set of all $W'$-conjugates of $\mu'$ that are dominant regular for $\frl'$. The set $Orb'$ is easy to write down - each element of $Orb'$ is given by dividing the coordinates of $\mu'$ into two groups of sizes $k-l$ and $n-k-l$, and then arranging each group in descending order.  
Now we can send any $\nu\in Orb$ to $\nu'\in Orb'$ obtained by deleting all repeated pairs from $\nu$. Conversely, we can start with $\nu'$ and add back the repeated pairs, one member of each pair into the first group of coordinates and one into the second group, each put into the right place to get descending order.

The following theorem summarizes the main results we prove in this paper.

\begin{thm}
\label{thmintro} 
Let $G=SL(n,\bbC)$, $\frg=\frsl(n,\bbC)$, and let $P$ be the parabolic subgroup of $G$ described above, for some $k\leq\frac{n}{2}$. Let $G(k,n)\cong G/P$ be the (complex) Grassmannian of $k$-dimensional
subspaces in $\bC^n.$ For every $\frp$-dominant integral weight $\la,$ let $F_\la$ denote the finite dimensional
$P$-module with lowest weight $-\la,$ and
let $\mathcal{O}_\frp(F_\la)$ denote the sheaf of sections of the homogeneous bundle on $G/P$ induced by $F_\la$.

Let $\mu$ be a dominant singular integral weight for $\frg$, such that every component of $\mu$ appears at most twice, and such that the number of repeated pairs is $l\leq k$.  Let $S$ be the singularity set (\ref{sing set}).

For any $s=0,\ldots,(k-l)(n-l-k),$ we define the chain space 
\[
C_s=\bigoplus_{\nu}\mathcal{O}_\frp(F_{\nu-\rho}),
\]
where the sum is over
all $\nu\in Orb$ with  $\nu'=w'(\mu')$, $\ell(w')=s$.
Here $\ell(w')$ is the length of $w'\in W'.$

Then for all $s=0,\ldots,(k-l)(n-l-k),$ there are invariant differential operators $D_s:C_s\to C_{s+1}$ such that on the big cell of $G/P$, 
the complex $(C_s,D_s)$ is a resolution of the 
$\gog$-module with lowest weight $-\mu+\rho$.

The degrees of the differential operators in the resolution depend  both on
the choice of the weight $\mu$ and on the choice of singularity set
 $S$ of a given degree $l.$
\end{thm}	

As in the regular cases, the simplest (bottom) case corresponds to the choice $\mu=\rho.$ In this case, the relative BGG resolution
 is fiberwise the de Rham complex (twisted by a line bundle), and all operators in it have order one. The orders of operators in the singular
 BGG sequences on $G/P$ are then equal to the order of the differential
 in the spectral sequences (i.e., to the number of steps used). It is well
 visible from examples presented in Section 2 that the Hasse diagrams
 for different singularity sets $S$ with the same order of singularity
 are the same but orders of the corresponding differential operators
 depend on the singularity set. It is also visible in these examples that 
 the types of operators also depend on $S$; sometimes they are standard
 (differential of order $1$ in the spectral sequence) and sometimes
 they are nonstandard (higher order differentials).

The main theorem is proved using the Penrose transform for a suitable choice of the twistor space $G/Q.$ The choice  of the parabolic subgroup $Q$ depends on the size $l$ of the singularity set $S.$ 
 The subgroup $Q$ is given by
\[
Q=\left\{\begin{pmatrix} l\times l & l\times (n-l) \cr
                                      0          & (n-l)\times (n-l)
           \end{pmatrix}\right\}.
\]
  The Levi subgroup of $Q$ is
$M=S(GL(l,\bbC)\times GL(n-l,\bbC))$, and there is now exactly one conjugate of 
$\mu$ that is dominant regular for $M$: the first group of $l$ coordinates must be $\mu_{s_1},\mu_{s_2},\dots,\mu_{s_l}$, i.e., one member from each repeated pair, and the second group contains other coordinates of $\mu$, in descending order.
 
Following \cite{BE}, we now use Penrose transform attached to the double fibration  
\medskip
\eq
\label{doublefib}
\xymatrix{
& & & & G/(P\cap Q) \ar[rd]^\tau\ar[ld]_\eta & & & &  \\
& & &  G/Q & & G/P.  & &  \\ 
}
\medskip
\eeq

The partial flag variety $G/P$ is the Grassmannian $G(k,n)$ of $k$-planes in 
$\bbC^n$, $G/Q$ is the Grassmannian $G(l,n)$ of $l$-planes in 
$\bbC^n$ and $G/(P\cap Q)$ is the partial flag variety $F(l,k,n)$ of flags 
\[
0\subset L_1\subset L_2\subset\bbC^n,\qquad \dim L_1=l,\ \dim L_2=k.
\]    

Let $U$ be the big cell in $G/P$. Let $V=\tau^{-1}(U)$ and let $W=\eta(V)$.
Then we have the double fibration
\bigskip
\eq
\label{doublefibcell}
\xymatrix{
& & & & V \ar[rd]^\tau\ar[ld]_\eta & & & &  \\
& & &  W & & U.  & &  \\ 
}
\eeq

The Penrose transform starts with the sheaf $\mathcal{O}_{\frqqq}(\mu)=\mathcal{O}_{\frqqq}(F_{\mu-\rho})$ of sections of the $G$-equivariant vector bundle on $G/Q$ corresponding to the finite-dimensional representation of $Q$ with lowest weight $-\mu+\rho$. 

The sheaf $\mathcal{O}_{\frqqq}(\mu)$ is restricted to $W$ and pulled back to $V$, by the sheaf-theoretic inverse image functor. Then the obtained sheaf on $V$ is resolved by the relative BGG resolution along the fiber, 
$\Delta^{\cdot}(\mu)$, which is in turn pushed down to $U$. The main fact about Penrose transform is the existence of the hypercohomology spectral sequence

\eq
\label{hyper}
E_1^{pq}=\Gamma(U,\tau_*^q\Delta^p(\mu)) \Rightarrow E_\infty^{pq}=H^{p+q}(W,\mathcal{O}_{\frqqq}(\mu)).
\eeq

We will study this spectral sequence in detail. In particular, we will construct all the (higher) differentials of this spectral sequence explicitly. A special feature will be that all of them can be defined on $E_1$, unlike in the usual situation when $d_r$ is only defined on $E_r$. 
We will then show that the complex $E_1$ with all the higher differentials put together into one differential gives the BGG resolution described in Theorem \ref{thmintro}. 

In Section \ref{sec:BGG} we explain in some detail the relative BGG resolutions mentioned above. In particular, we identify objects of our resolution-to-be.
In Section \ref{sec:higher} we define the higher differentials of the spectral sequence, all at the level of $E_1$, and so we get our singular BGG complex. Finally, in Section \ref{sec:exactness} we prove the exactness of this complex on 
the big cell. It is in fact true that the complex is exact as a complex of sheaves, but the proof of this fact requires additional techniques and we postpone it to a future publication.

In this paper we only consider type A groups. The conformal case was settled a while ago \cite{B}, and type C is being considered by Rafael Mr\dj en (in preparation).

\noindent{\bf Acknowledgement.} The authors would like to thank Mike Eastwood for a discussion, comments and useful hints concerning the topic of the paper.

\section{Relative BGG resolutions and their direct images}
\label{sec:BGG}

The main idea of the application of the Penrose transform for construction of BGG sequences in singular infinitesimal character
can be explained as follows. Fibers of the projection $\eta$ to the twistor space are again flag manifolds. Sections of the pullback of the sheaf $\mathcal{O}_{\frqqq}(\mu)$ from the subset $W$ in the twistor
space to $V$ in the correspondence space are constant along the fibers,
hence it is possible to consider their fiberwise resolutions
along individual fibers using the BGG resolutions for finite dimensional modules, suitably glued together over the base space. This is what is done by the relative BGG resolutions.
A description of the algorithm for computing of the relative BGG complex is given in \cite{BE}.
  A detailed
treatment of relative BGG resolutions in geometric form (even
in the curved version) can be found in recent preprints \cite{CS1,CS2}.

While the bundles chosen on the twistor space (and all other
bundles used in the process of the Penrose transform) can have singular
infinitesimal character, relative BGG resolutions for finite dimensional representations used 
in the fiberwise  resolutions are in regular infinitesimal character
and the standard BGG constructions for them can be used. So the construction for BGG resolution in regular cases is used in the machine
of the Penrose transform to get resolutions in singular infinitesimal
character.

Before going to the general case, we prefer  to illustrate the behaviour of the machine in some simple examples.

\begin{ex}
	{\rm 
 We first discuss an example of an orbit in the regular BGG sequence
 for the Grassmannian $G(2,6).$ So let  $G=SL(5,\bbC)$
 with the Levi subgroup of $P$ equal to $L=S(GL(2,\bbC)\times GL(3,\bbC))$. 
Take the parameter  $\mu=(43|210).$
The two groups of coordinates separated by the bar correspond to the two factors of $L$. The points of the orbit (the Hasse diagram) are the conjugates of $\mu$ which are dominant for $L$, i.e., the coordinates in each of the groups are decreasing:

$$
\begin{matrix} 
(43|210)& \rightarrow & (42|310)& \rightarrow & (41|320)& \rightarrow & (40|321) \\ 
    & & \downarrow & & \downarrow & & \downarrow \\
& & (32|410)&\rightarrow & (31|420)& \rightarrow & (30|421) \\
& & & & \downarrow & & \downarrow \\
& & & & (21|430)& \rightarrow & (20|431) \\
& & & & & & \downarrow  \\
& & & & & &( 10|432)\\
\end{matrix}
$$

The arrows, denoting immediate successors, always connect two conjugates of 
$\mu$ that can be obtained from each other by a single transposition across the bar. 

To obtain the algebraic (i.e., the category $\cO$) version of the BGG resolution from the Hasse diagram, we introduce the degree by letting the first point have degree zero, and other points have degree equal to the number of arrows needed to get to the point from the first point. Then we add up the Verma modules corresponding to the points in the diagram of each degree, and the morphisms follow the arrows of the diagram, but in the opposite direction. The geometric (sheaf) version of the BGG resolution is constructed analogously, but now the morphisms go in the same direction as the arrows in the diagram.
}
\end{ex}

We shall now present further examples of relative BGG sequences and we shall discuss their direct images.
 
In general, we choose the twistor space $G/Q$ with the Levi subgroup $L'$ of $Q$ equal to 
$S(GL(l)\times GL(n-l))$, in such a way that $l$ is exactly the number of repeated coordinates of 
the parameter $\mu$. Then there is only one conjugate of $\mu$,
\[
\tilde\mu= (\tilde\mu_1\,\big|\, \tilde\mu_2),
\]
which is dominant regular for $L'$. It is determined by the fact that the first group of coordinates $\tilde\mu_1$ consists of exactly one coordinate from each repeated pair. 

The terms of the relative BGG resolution of the pullback sheaf correspond to conjugates 
of $\mu$ with three groups of coordinates, the first group being always equal to 
$\tilde\mu_1$. The other two groups of coordinates are of sizes $k-l$ and $n-k$, and they correspond to the regular BGG sequence for $\tilde\mu_2$. 

\begin{ex} 
{\rm 
Let $n=8$, $k=4$ and $l=2$. The double fibration (\ref{doublefib}) can be described by

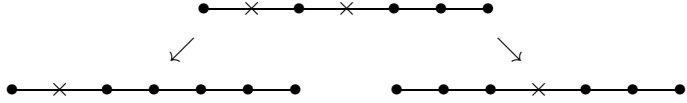
\[
\dynkin \root{ }\link\noroot{ }\link\root{ }\link\noroot{ }
\link\root{ }\link\root{ }\link\root{ }
\enddynkin
\]
\[
\swarrow\hskip 4cm\searrow
\]
\[
\dynkin \root{ }\link\noroot{ }\link\root{ }\link\root{ }
\link\root{ }\link\root{ }\link\root{ }
\enddynkin
\hskip 1 cm
\dynkin \root{ }\link\root{ }\link\root{ }\link\noroot{ }
\link\root{ }\link\root{ }\link\root{ }
\enddynkin
\] 
\vskip 5mm

The meaning of this is that $G/P$ corresponds to the standard parabolic subgroup $P$ which is obtained from $G=SL(8)$ by crossing the fourth point of the Dynkin diagram, i.e., the roots for the Levi factor of $P$ are generated by the simple roots for $G$ other than the fourth one. Similarly, $G/Q$ corresponds 
to $Q$ which is obtained from $G$ by crossing the second point of the Dynkin diagram, and $G/(P\cap Q)$ corresponds to $P\cap Q$ which is obtained from $G=SL(8)$ by crossing both the second and the fourth point of the Dynkin diagram.

As we said above, choosing $l=2$ corresponds to considering $\mu$ with two repeated coordinates. For example, $\mu$ could be 
$(55432210)$. Then the relative BGG resolution is
 
\[
\begin{matrix}
(52|54|3210)& \hskip -2mm \rightarrow \hskip -2mm &(52|53|4210)& \hskip -2mm \rightarrow \hskip -2mm &(52|52|4310)& \hskip -2mm \rightarrow \hskip -2mm &(52|51|4320) \hskip -2mm &\rightarrow \hskip -2mm &(52|50|4321)\\
&& \downarrow && \downarrow &&\downarrow &&\downarrow \\
& &(52|43|5210)& \hskip -2mm \rightarrow \hskip -2mm &(52|42|5310)& \hskip -2mm \rightarrow \hskip -2mm &(52|41|5320)
& \hskip -2mm \rightarrow \hskip -2mm &(52|40|5321)\\
&&&& \downarrow &&\downarrow &&\downarrow \\
 && && (52|32|5410)& \hskip -2mm \rightarrow \hskip -2mm &(52|31|5420) &\hskip -2mm \rightarrow \hskip -2mm &(52|30|5421)\\
 &&&&&& \downarrow &&\downarrow \\
 &&&&&& (52|21|5430) & \hskip -2mm \rightarrow \hskip -2mm &(52|20|5431)\\
 &&&&&&&& \downarrow \\
 &&&&&&&&(52|10|5432)\\
\end{matrix}
\]
\medskip

As usual, the direct images are computed using the Bott-Borel-Weil Theorem. 
In practice, this means we are deleting the first of the two bars in each of the $W$-conjugates of $\mu$ in the Hasse diagram. After this, the second group is automatically dominant regular for $L$, but the first group need not be dominant or regular. By the
Bott-Borel-Weil theorem, if the first group is not regular, then all cohomology groups of the corresponding sheaf are zero. If the first group is regular, there is exactly one nonvanishing cohomology group, and it appears in degree equal to the length of the Weyl group element required to make the corresponding conjugate of $\mu$ dominant.

In the above example, the conjugates of $\mu$ which have the group of first four coordinates regular are $(5243|5210)$, $(5241|5320)$, $(5240|5321)$, $(5231|5420)$,
$(5230|5421)$, and $(5210|5432)$. The respective cohomology degrees are 2, 1, 1, 1, 1, and 0. (For example, to permute 5243 to descending order we need two transpositions of neighbours, so the degree is two.)
 
A short description of these direct images is 
$$
\begin{matrix}
x&x&x&x&x\\
&2&x&1&1\\
&&x&1&1\\
&&&x&x\\
&&&&0\\
\end{matrix}
$$
Here $x$ means there is no cohomology at the corresponding point, and each number shows the cohomology degree of the corresponding point.

We can similarly analyze other cases; for example, if the dominant weight is
$\mu=(55443210)$, then the relative Hasse diagram is 

\[
\begin{matrix}
(54|54|3210)&\hskip -2mm \rightarrow \hskip -2mm &(54|53|4210)& \hskip -2mm \rightarrow \hskip -2mm &(54|52|4310)& \hskip -2mm \rightarrow \hskip -2mm &(54|51|4320)
&\hskip -2mm \rightarrow \hskip -2mm &(54|50|4321)\\
&& \downarrow && \downarrow &&\downarrow &&\downarrow \\
& &(54|43|5210)& \hskip -2mm \rightarrow \hskip -2mm &(54|42|5310)& \hskip -2mm \rightarrow \hskip -2mm &(54|41|5320)
&\hskip -2mm \rightarrow \hskip -2mm &(54|40|5321)\\
&&&& \downarrow &&\downarrow &&\downarrow \\
&& && (54|32|5410)& \hskip -2mm \rightarrow \hskip -2mm &(54|31|5420)
&\hskip -2mm \rightarrow \hskip -2mm &(54|30|5421)\\
&&&& &&\downarrow &&\downarrow \\
&&& &&  & (54|21|5430)
&\hskip -2mm \rightarrow \hskip -2mm &(54|20|5431)\\
&&&& &&&&\downarrow \\
&&&& &&  &  
&(54|10|5432)\\
\end{matrix}
\]
  
The short description of the direct images is

\[
\begin{matrix}
x&x&x&x&x\\
&x&x&x&x\\
&&0&0&0\\
&&&0&0\\
&&&&0\\
\end{matrix}
\]

For $\mu=(54321100)$, the relative Hasse diagram is 

\[
\begin{matrix}
(10|54|3210)&\hskip -2mm \rightarrow \hskip -2mm &(10|53|4210)& \hskip -2mm \rightarrow \hskip -2mm &(10|52|4310)& \hskip -2mm \rightarrow \hskip -2mm &(10|51|4320)
&\hskip -2mm \rightarrow \hskip -2mm &(10|50|4321)\\
&& \downarrow && \downarrow &&\downarrow &&\downarrow \\
 &&(10|43|5210)& \hskip -2mm \rightarrow \hskip -2mm &(10|42|5310)& \hskip -2mm \rightarrow \hskip -2mm &(10|41|5320)
&\hskip -2mm \rightarrow \hskip -2mm &(10|40|5321)\\
&&&& \downarrow &&\downarrow &&\downarrow \\
&& && (10|32|5410)& \hskip -2mm \rightarrow \hskip -2mm &(10|31|5420)
&\hskip -2mm \rightarrow \hskip -2mm &(10|30|5421)\\
&&&& &&\downarrow &&\downarrow \\
&&& &&  & (10|21|5430)
&\hskip -2mm \rightarrow \hskip -2mm &(10|20|5431)\\
&&&& &&&&\downarrow \\
 &&&&&&  &  
&(10|10|5432)\\
\end{matrix}
\]

and the short description of the direct images is

\[
\begin{matrix}
4&4&4&x&x\\
&4&4&x&x\\
&&4&x&x\\
&&&x&x\\
&&&&x\\
\end{matrix}
\]

For $\mu=(54432110)$, the relative Hasse diagram is 

\[
\begin{matrix}
(41|54|3210)&\hskip -2mm \rightarrow \hskip -2mm &(41|53|4210)& \hskip -2mm \rightarrow \hskip -2mm &(41|52|4310)& \hskip -2mm \rightarrow \hskip -2mm &(41|51|4320)
&\hskip -2mm \rightarrow \hskip -2mm &(41|50|4321)\\
&& \downarrow && \downarrow &&\downarrow &&\downarrow \\
& &(41|43|5210)& \hskip -2mm \rightarrow \hskip -2mm &(41|42|5310)&\hskip -2mm \rightarrow \hskip -2mm & (41|41|5320)
&\hskip -2mm \rightarrow \hskip -2mm &(41|40|5321)\\
&&&& \downarrow &&\downarrow &&\downarrow \\
 &&&& (41|32|5410)& \hskip -2mm \rightarrow \hskip -2mm &(41|31|5420)
&\hskip -2mm \rightarrow \hskip -2mm &(41|30|5421)\\
&&&& &&\downarrow &&\downarrow \\
&&& &&  & (41|21|5430)
&\hskip -2mm \rightarrow \hskip -2mm &(41|20|5431)\\
&&&& &&&&\downarrow \\
&&&& &&  &  
&(41|10|5432)\\
\end{matrix}
\]

and the short description of the direct images is

\[
\begin{matrix}
x&3&3&x&2\\
&x&x&x&x\\
&&2&x&1\\
&&&x&1\\
&&&&x\\
\end{matrix}
\]
}
\end{ex}

We can see from these examples that while the relative BGG resolution corresponds to the regular BGG resolution with one member from each pair of repeated coordinates removed, the part of the diagram with non-vanishing direct images corresponds to the regular BGG resolution with both repeated coordinates removed. Namely, while one member from each repeated pair does not change  in the first group of coordinates, the second member must be without a change  in the last group of coordinates. In this way we have moved to the Enright-Shelton category mentioned in the introduction.

\begin{rem}
{\rm

To describe the situation in general, let us denote the repeated coordinates of 
$\mu$ by $i_1>\dots>i_l\geq 0$, and the non-repeated coordinates by
$j_1>\dots>j_{n-2l}\geq 0$.  

Let
\[
I=\{i_1,\dots,i_l\};\qquad J=\{j_1,\dots,j_{n-2l}\}.
\]
The relative Hasse diagram has vertices of the form
\[
(i_1\dots i_l|k_1\dots k_a|l_1\dots l_b),\qquad k_1,\dots,k_a,l_1,\dots,l_b\in I\cup J, 
\]
with $k_1>\dots>k_a$ and $l_1>\dots>l_b$. Here $a=k-l$ and $b=n-k$.

Such a vertex will give nonzero direct image if and only if $k_1,\dots,k_a\in J$. By the Bott-Borel-Weil 
Theorem, the degree of the nonvanishing direct image is the number of transpositions of neighboring
coordinates necessary to bring $i_1\dots i_l k_1\dots k_a$ to descending order. As in the above examples,
we label such a vertex by the corresponding degree, while the vertices with zero direct image are labeled 
by $x$.

The arrows in the relative Hasse diagram go from $(i_1\dots i_l|k_1\dots k_a|l_1\dots l_b)$ to every 
vertex that can be obtained from it by a transposition of some $k_r$ and some $l_s$. For a fixed $r$, such an
arrow is possible precisely when there are $s$ and $u$ such that $u\geq s$ and
\[
l_{s-1}>k_r>l_s>\dots>l_u>k_{r+1}>l_{u+1}.
\]
We set $k_{a+1}=-1$ to include the case $r=a$. In this situation, there are the following arrows in the relative Hasse diagram,
all pointing ``in the $r$th direction":
\begin{multline*}
(i_1\dots i_l|k_1\dots k_{r-1} k_r k_{r+1}\dots k_a|l_1\dots l_{s-1}l_s l_{s+1}\dots l_u\dots l_b)\to \\
\to (i_1\dots i_l|k_1\dots k_{r-1} l_s k_{r+1}\dots k_a|l_1\dots l_{s-1}k_r l_{s+1}\dots l_u\dots l_b)\to \\
\to (i_1\dots i_l|k_1\dots k_{r-1} l_{s+1} k_{r+1}\dots k_a|l_1\dots l_{s-1} k_r l_s l_{s+2} \dots l_u \dots l_b)\to\dots \\
\dots\to (i_1\dots i_l|k_1\dots k_{r-1} l_u k_{r+1}\dots k_a|l_1\dots l_{s-1}k_r l_s l_{s+1}\dots l_{u-1} l_{u+1}\dots l_b).
\end{multline*}
We say that these arrows point in the $r$th direction because the $r$th component in the second group is changing, while the other components stay without a change.
}
\end{rem}

\begin{prop} 
\label{jumps}
In the above situation, assume that $(i_1\dots i_l|k_1\dots k_a|l_1\dots l_b)$ is labeled by a number $\alpha$, and assume that for some $v$ such that $s<v\leq u$,
$l_s,\dots l_{v-1}\in I$, while $l_v\in J$. Then the vertex
\[
(i_1\dots i_l k_1\dots k_{r-1} l_v k_{r+1}\dots k_a|l_1\dots l_{s-1} k_r \dots l_{v-1} l_{v+1} \dots l_b),
\]
obtained by deleting the first bar, is labeled by $\alpha-(v-s)$, while the $v-s$ vertices in between these two are all labeled by $x$.
\end{prop}

\begin{proof} It is clear that the vertices in between are labeled by $x$. Furthermore, in rearranging 
\[
(i_1\dots i_l k_1\dots k_{r-1} l_v k_{r+1}\dots k_a) 
\]
into descending order, $l_v$ does not have to cross over $l_s,\dots,l_{v-1}$,
which form a string inside $i_1\dots i_l$, while in rearranging 
\[
(i_1\dots i_l k_1\dots k_{r-1} k_r k_{r+1}\dots k_a)
\] 
into descending order, $k_r$ does have to cross over  $l_s,\dots,l_{v-1}$.
So we see that rearranging 
\[
(i_1\dots i_l k_1\dots k_{r-1} l_v k_{r+1}\dots k_a)
\] 
into descending order requires $v-s$ fewer
transpositions than rearranging 
\[
(i_1\dots i_l k_1\dots k_{r-1} k_r k_{r+1}\dots k_a),
\] 
so the label drops by $v-s$, as claimed.
\end{proof}

\begin{rem}
\label{enright-shelton}
{\rm
We also see from the above considerations that, as in the examples we have seen, the relative BGG resolution follows the pattern of the regular BGG resolution with one member from each pair of repeated coordinates of $\mu$ removed, while the points with nonvanishing direct image correspond to the terms of the regular BGG resolution with both repeated coordinates removed (the BGG resolution in the Enright-Shelton category). Our next task is to organize the points with nonvanishing direct image into a resolution. After we do this, we will see that this resolution is dual to the BGG resolution in the Enright-Shelton category.
}
\end{rem}

\section{Definition of the higher differentials}
\label{sec:higher}

\bigskip

We denote by $d_0$ the ``vertical" differential of the \v Cech complex, which we use to compute the Bott-Borel-Weil cohomology, and by 
$d_1$ the ``horizontal" differential of the relative BGG resolution. 
At every point
of the Hasse diagram corresponding to the relative BGG resolution, $d_1$ 
has components $d_1^{(r)}$  in the directions of the edges of the Hasse diagram 
emanating from that point.

Let $[a]$ be a vertical cohomology class at a point of the above Hasse diagram 
(by vertical cohomology we mean cohomology with respect to $d_0$). Suppose
that the next $i-1$ points in the $r$th direction from $a$ have no cohomology, at any vertical place of the corresponding column. Then we define
$d_i^{(r)} [a]$ as follows. 

Let $a_1=d_1^{(r)} a$. Since $d_0 a=0$, it follows that 
\[
d_0 a_1 =d_0 d_1^{(r)} a = -d_1^{(r)}d_0 a=0.
\]
Since we assumed there is no cohomology at the point where $a_1$ lies, it follows that
there is $a_2$ such that $d_0a_2=a_1$. Let $a_3=d_1^{(r)} a_2$. Then
\[
d_0a_3=d_0 d_1^{(r)} a_2=-d_1^{(r)} d_0 a_2 =-d_1^{(r)} a_1=-d_1^{(r)} d_1^{(r)} a =0.
\]
Since there is no cohomology at the point where $a_3$ lies, we can find $a_4$ such that
$d_0 a_4=a_3$. Proceeding like this, we find 
\[
a_0=a,a_1,a_2,\dots,a_{2i-1}
\]
with the properties

\begin{equation}
\label{diseq}
d_1^{(r)} a_{2k}=a_{2k+1}=d_0 a_{2k+2},\qquad k=0,1,\dots,i-2;\qquad d_1^{(r)}a_{2i-2}=a_{2i-1}.
\end{equation}

The corresponding picture is

\[
\begin{CD}
a_0 @>>> a_1  @. \\
@.  @AAA   @. \\
@. a_2 @>>> a_3  \\ 
@.@.@AAA \\
@. @. \cdots  \\
@.@.@.  @>>> a_{2i-3} \\
@.@.@.@.@AAA \\
@.@.@.@. a_{2i-2} @>>> a_{2i-1},
\end{CD}
\]

with vertical arrows representing $d_0$ and horizontal arrows representing $d_1^{(r)}$.
Note that while we do get $d_0 a_{2i-1}=0$, we can not proceed to find $a_{2i}$ with 
$d_0 a_{2i}=a_{2i-1}$, since we are not assuming that there is no cohomology at the point
where $a_{2i-1}$ lies. 

We define $d_i^{(r)}[a]=[a_{2i-1}]$, and we define $d_i[a]$ by putting together all the components
$d_i^{(r)}[a]$.

\begin{lem} 
\label{defdi}
Let $a_0,\dots,a_{2i-1}$ and $a_0',\dots,a_{2i-1}'$ be two sequences as above, with $[a_0]=[a_0']$, i.e.,
$a_0'-a_0=d_0b_0$ for some $b_0$. 
In particular, both sequences satisfy (\ref{diseq}), and we are assuming that there is no vertical cohomology at
places of $a_1,a_3,\dots,a_{2i-3}$.

Then $[a_{2i-1}]=[a_{2i-1}']$.
In other words, the above definition of $d_i$ is good, i.e., independent of all the choices.
\end{lem}
\begin{proof} We prove by induction on $k$ that $a_{2k-1}'=a_{2k-1}+d_0b_{2k-1}$ for some $b_{2k-1}$ in the image of
$d_1^{(r)}$. In particular, for $k=i$ we obtain the statement of the lemma.
If $k=1$, then 
\[
a_1'=d_1^{(r)}a_0'=d_1^r(a_0+d_0b_0)=a_1-d_0d_1^{(r)}b_0,
\]
and the claim is true with $b_1=-d_1^{(r)}b_0$.

Assume that the claim holds for some $k$, $1\leq k<i$, so $a_{2k-1}'=a_{2k-1}+d_0b_{2k-1}$, with $b_{2k-1}$ in the image of
$d_1^{(r)}$. Then
\[
d_0a_{2k}'=a_{2k-1}'=a_{2k-1}+d_0b_{2k-1}=d_0(a_{2k}+b_{2k-1}).
\]
So $d_0(a_{2k}'-a_{2k}-b_{2k-1})=0$, and since there is no vertical cohomology at this place, we conclude
\[
a_{2k}'-a_{2k}-b_{2k-1}=d_0b_{2k},
\]
for some $b_{2k}$. Applying $d_1^{(r)}$, we obtain
\[
a_{2k+1}'-a_{2k+1}-d_1^{(r)}b_{2k-1}=d_1^{(r)}d_0b_{2k}=-d_0d_1^{(r)}b_{2k}.
\]
Since $b_{2k-1}$ is in the image of $d_1^{(r)}$, $d_1^{(r)}b_{2k-1}=0$, and we obtain the claim for $k+1$, with
$b_{2k+1}=-d_1^{(r)}b_{2k}$. This finishes the induction.
\end{proof}

\begin{df}
We now define $d[a]$ as the sum of all $d_i[a]$.	
\end{df}

  We want to see that $d$ defines the singular BGG
resolution we wanted to obtain.

\begin{prop} 
\label{diff}
The above defined $d$ is a differential, i.e., $d^2=0$.
\end{prop}

\begin{proof} It is enough to prove that for any two directions $r,s$ at a point of the Hasse diagram of the fiber,
and any $i,j\geq 1$, 
\[
d_i^{(r)}d_j^{(s)}+d_j^{(s)}d_i^{(r)}=0.
\]
This is clear if $r=s$, so we assume in the following that $r\neq s$.

We will call any sequence $a_0=a,a_1,a_2,\dots,a_{2i-1}$ satisfying (\ref{diseq}) a $d_i^{(r)}$-sequence.
Here we do not assume anything about vanishing of cohomology; if we however assume vanishing as in the 
definition of $d_i$, then we see from Lemma \ref{defdi} that a $d_i^{(r)}$-sequence starting with a cycle 
$a_0$ always exists, and that any such sequence can be used to define $d_i^{(r)}[a_0]$, as $[a_{2i-1}]$.  

We first prove the following statement about pushing out $d_i^{(r)}$-sequences.

\begin{lem}[\bf one-step push forward]
\label{pushout}
Let $a_0,\dots,a_{2i-1}$ be any $d_i^{(r)}$-sequence. Let $s\neq r$, and set 
\[
b_k=(-1)^k d_1^{(s)} a_k,\qquad k=0,1,\dots 2i-1.
\]
Then $b_0,\dots,b_{2i-1}$ is a $d_i^{(r)}$-sequence.
\end{lem}
\begin{proof} If $1\leq j\leq i-1$, then
\[
d_0b_{2j}=d_0d_1^{(s)}a_{2j}=-d_1^{(s)}d_0a_{2j}=-d_1^{(s)}a_{2j-1}=b_{2j-1}.
\]
Furthermore, if $0\leq j\leq i-1$, then
\[
d_1^{(r)}b_{2j}=d_1^{(r)}d_1^{(s)}a_{2j}=-d_1^{(s)}d_1^{(r)}a_{2j}=-d_1^{(s)}a_{2j+1}=b_{2j+1}.
\]
\end{proof}

We can now prove Proposition  \ref{diff} in case $j=1$. Namely, let $a=a_0,a_1,\dots,a_{2i-1}$ be a  $d_i^{(r)}$-sequence
computing  $d_i^{(r)}a$. Then, by Lemma \ref{pushout}, $b_k=(-1)^k d_1^{(s)} a_k$ defines a  $d_i^{(r)}$-sequence computing 
$d_i^{(r)} d_1^{(s)} a$. In particular, 
\[
b_{2i-1}=-d_1^{(s)} a_{2i-1}=-d_1^{(s)} d_i^{(r)} a
\] 
is equal to $d_i^{(r)} d_1^{(s)} a$ and the claim follows.
Let us now assume that both $i$ and $j$ are bigger than 1. We need another lemma about pushing out sequences.

\begin{lem}[\bf two-step push forward]
\label{pushout2}
Let $a_0,\dots,a_{2i-1}$ be a $d_i^{(r)}$-sequence. Assume that for each $k$, $a_k=d_1^{(s)}a_k'$ for some $a_k'$ (here $s\neq r$).
Assume further that $a_0=d_0 b_0$ for some $b_0$, and that there is no vertical cohomology at the positions of $a_2,a_4,\dots, a_{2i-2}$. 
Then there is a $d_i^{(r)}$-sequence $c_0,\dots,c_{2i-1}$, with $c_0=d_1^{(s)}b_0$,
and such that there is a $b_{2i-1}$ with $d_0b_{2i-1}=a_{2i-1}$ and $d_1^{(s)}b_{2i-1}=c_{2i-1}$.
\end{lem}
\begin{proof} We set $c_0=d_1^{(s)} b_0$ and $c_1=d_1^{(r)}c_0$. Since
\[
d_0d_1^{(r)}b_0=-d_1^{(r)}d_0b_0=-d_1^{(r)}a_0=-a_1=-d_0a_2,
\]
we see that $d_0(d_1^{(r)}b_0+a_2)=0$. Since there is no vertical cohomology at the point where $a_2$ lies, there is a $b_2$ such that
\[
d_0b_2=d_1^{(r)}b_0+a_2.
\]
Let $c_2=d_1^{(s)}b_2$. Then
\[
d_0c_2=d_0d_1^{(s)}b_2=-d_1^{(s)}d_0b_2=-d_1^{(s)}d_1^{(r)}b_0-d_1^{(s)}a_2=d_1^{(r)}d_1^{(s)}b_0=d_1^{(r)}c_0=c_1.
\]
(Here we used the assumption that $a_2=d_1^{(s)}a_2'$, which implies $d_1^{(s)}a_2=0$.)

We see that $c_0,c_1,c_2$ is the beginning of a $d_i^{(r)}$-sequence. We can continue the construction in the same way, by checking
$d_0(d_1^{(r)}b_2+a_4)=0$, choosing $b_4$ such that $d_0 b_4=d_1^{(r)}b_2+a_4$, and setting $c_3=d_1^{(r)}c_2$, $c_4=d_1^{(s)}b_4$.
Then we check $d_0c_4=c_3$, and so we have defined our sequence $c_k$ up to $k=4$. We continue like this up to constructing
$c_{2i-2}=d_1^{(s)}b_{2i-2}$, where $d_0b_{2i-2}=d_1^{(r)}b_{2i-4}+a_{2i-2}$. So we have
\[
d_0c_{2i-2}=c_{2i-3}=d_1^{(r)} c_{2i-4}.
\]
We finish by setting
\[
c_{2i-1}=d_1^{(r)}c_{2i-2}, \qquad b_{2i-1}=-d_1^{(r)}b_{2i-2}.
\] 
It is now clear that $c_0,\dots,c_{2i-1}$ is a $d_i^{(r)}$-sequence. On the other hand, 
\[
d_0b_{2i-1}= -d_0d_1^{(r)}b_{2i-2}=d_1^{(r)}d_0b_{2i-2}=d_1^{(r)}a_{2i-2}=a_{2i-1},
\]
and
\[
d_1^{(s)}b_{2i-1}=-d_1^{(s)}d_1^{(r)}b_{2i-2}=d_1^{(r)}d_1^{(s)}b_{2i-2}=d_1^{(r)}c_{2i-2}=c_{2i-1},
\]
as required.
\end{proof}

We can now finish the proof of Proposition \ref{diff} in the general case. We start with a $d_i^{(r)}$-sequence $a=a_0,a_1,\dots,a_{2i-1}$ computing
$d_i^{(r)}a$.  We also have a $d_j^{(s)}$-sequence $a=x_0,x_1,\dots,x_{2j-1}$ computing $d_j^{(s)}a$. We first push forward the sequence $a_k$ using Lemma \ref{pushout}, to get a $d_i^{(r)}$-sequence from $x_1$ to $-d_1^{(s)}a_{2i-1}$. We continue by doing the two-step push forward of Lemma \ref{pushout2} $j-1$ times. Here at the $i$th application of Lemma \ref{pushout2}, the corresponding $b_0$ is $x_{2i}$.

We end up getting a $d_i^{(r)}$-sequence that can be used to compute $d_i^{(r)}x_{2j-1}=d_i^{(r)}d_j^{(s)}a$. On the other hand, the right ends of the pushed out sequences form a $d_j^{(s)}$-sequence that can be used to compute $d_j^{(s)}(-a_{2i-1})=-d_j^{(s)}d_i^{(r)}a$.
This implies that $d_j^{(s)}d_i^{(r)}a=-d_i^{(r)}d_j^{(s)}a$, as claimed.
\end{proof} 

We note that our construction also gives the higher differentials of the hypercohomology spectral sequence (\ref{hyper}). Namely, these higher differentials are also induced by $d_0$ and $d_1$, in the same way as in our construction. The special feature of our situation is that all the higher differentials are constructed on the first sheet $E_1$ of the spectral sequence, while in general $d_r$ can only be defined on $E_r$.

To see that the degree in the singular BGG complex we have constructed is well defined, one can use the fact that the points of the relative Hasse diagram involved in the singular BGG complex follow the pattern of the regular Hasse diagram for the Enright-Shelton category, as was explained in Section \ref{sec:BGG} (see Remark \ref{enright-shelton}).

\begin{lem}
The differentials $d_i$ are differential operators. Their orders
depend on parameters $\mu$ as well as on the choice of the singularity
set $S.$ 
\end{lem}

\begin{proof}
	
The construction of the Penrose transform was shortly reviewed   in Section \ref{sec:intro}. Instead of the big cell, it is possible to start with
any (small) ball $U$ in it,  and to define $V=\tau^{-1}(U)$ and $W=\eta(V).$ So the operators given by the differential in the resulting complex are local operators. Hence it is well known (see e.g. \cite{KMS}) that they are given
by differential operators, possibly of infinite order. Hence they are induced by homomorphisms from the jet bundles (possibly of infinite order). But then their order
is controlled by their generalized conformal weights. These are defined as follows. Suppose that we restrict the differential $d$ to one irreducible
piece (i.e., to the bundle induced by an irreducible $L$-module $V_{\nu_1}$ with values in another irreducible piece $V_{\nu_2}$). Suppose that
$E\in\gog_0$ is the grading element for $\gog.$ Then the order
of the corresponding differential operators is given by the difference
$\nu_1(E)-\nu_2(E).$
\end{proof}

\section{Exactness of the BGG complex}
\label{sec:exactness}

In this section we will prove that our singular BGG complex is exact except in degree 0, so we have obtained a resolution of the kernel of the first differential operator in the complex. More precisely, we will prove exactness on the big cell only, not at the sheaf level.

Recall the double fibration (\ref{doublefibcell})
and the hypercohomology spectral sequence (\ref{hyper}).
By remarks from the end of Section \ref{sec:higher}, passing through this spectral sequence gives exactly the global sections of the cohomology of our singular BGG complex, with a shift in degree.  

The shift in degree comes from the fact that the degree 0 point in our singular BGG complex is of degree $p$ inside the relative BGG resolution, and has vertical degree $q$ (specified by the number written over the point), with $p+q=l(k-l)$. (Recall that the Levi subgroup of $P$ is $S(GL(k)\times GL(n-k))$, while 
the Levi subgroup of $Q$ is $S(GL(l)\times GL(n-l))$.) 

We prove this in the following Lemma.

\begin{lem}
\label{shift}
The smallest $p+q$ for which $\Gamma(X,\tau_*^q\Delta^p(\mu))\neq 0$ is equal to $l(k-l)$.
\end{lem}
\begin{proof} The starting weight in the relative BGG resolution is 
\eq
\label{startwt}
(i_1\dots i_l|k_* i_1 k_* i_2\dots i_r k_*|l_1 l_2 \dots)
\eeq
Here we denoted by $k_*$ the groups of non-repeated coordinates separating the repeated coordinates $i_1,\dots,i_r$. Let the sizes of these groups be respectively $x_1,x_2,\dots,x_{r+1}$; each $x_i$ is between 0 and $k'=k-l$, and their sum is
\[
x_1+\dots+x_{r+1} = k'-r.
\]
Now $p$ is the (minimal) number of arrows in the Hasse diagram that are required to free the second group from all the $i$s. After this is achieved, $q$ is  the (minimal) number of transpositions of neighbors in the union of the first two groups required to bring the coordinates into descending order.

We prove the assertion $p+q=lk'$ by induction on $l$; here we work for all 1-graded maximal parabolic subgroups in all $SL(n)$ simultaneously, i.e., the size of the weights need not stay the same throughout the argument.

We start with $l=0$. In this case it is clear that $p=q=0$, so the assertion trivially holds.

Assume now that we know the assertion in case there are $l-1$ repeated coordinates, and let us prove it in case there are $l$ repeated coordinates. The starting weight is as in (\ref{startwt}). In case $r=0$, i.e., there are no $i$ coordinates in the second group, we have $p=0$, and $q=lk'$ as all the $k$s need to be permuted in front of all the $i$s. Thus the assertion holds in this case.

Suppose now $r\geq 1$, and rewrite (\ref{startwt}) as
\eq
\label{startwt1}
(i_1\dots i_l|k_1\dots k_{x_1}i_1 u_1\dots u_{k'-x_1-1}|l_1 l_2\dots)
\eeq
We have change the notation because some of the $u_i$ can be among the $k$s and some among the $i$s, but we do not need to distinguish between them 
at this point.

In order to free the second group of the $i$s, we may start by moving $i_1$ into the third group. This requires $k'-x_1$ moves: we first switch $u_{k'-x_1-1}$ with $l_1$, then we switch $u_{k'-x_1-1}$ with $u_{k'-x_1-2}$, and so on, until we switch $i_1$ with $u_1$.
After this we can still have a number of $i$s in the second group (either $r-1$ or $r$, depending on $l_1$), but all of them are smaller than $i_1$ which is now
in the first position of the third group. Moving all these $i$s into the third group will be the same as the analogous moving if we delete $i_1$ from both of its
current positions. This allows us to use the inductive assumption, and conclude that we need $(l-1)k'$ moves to free the second group of the $i$s and then permute the union of the first two groups into descending order. After this we return $i_1$ to both of its positions; now we still need to move
$k_1,\dots,k_{x_1}$, which were not touched by any of the previous moves, past $i_1$. This requires $x_1$ moves, and so the total number of required moves is
\[
k'-x_1+(l-1)k'+x_1 = lk',
\]
as claimed.
\end{proof}

Since the big cell $U$ is Stein, taking global sections commutes with taking cohomology of a complex. So it is enough to see that $E_\infty^{pq}$ vanishes in the degrees above $l(k-l)$. 

This will follow once we prove the following theorem.

\begin{thm}
\label{vanishingcoho}
The cohomology of all coherent holomorphic sheaves on $W$ vanishes above the degree $l(k-l)$.
\end{thm}

Before proving Theorem \ref{vanishingcoho}, we recall some standard notation and conventions about Grassmannians. The Grassmannian of $l$-planes in $\bbC^k$, denoted by $G(l,k)$, is a smooth projective variety. The Stiefel variety $S(l,k)$ is a principal $GL(l,\bbC)$-bundle over $G(l,k)$, with elements that can be idenitified with $k\times l$ matrices of rank $l$, and the $GL(l,\bbC)$-action is given by right multiplication. The projection from $S(l,k)$ to $G(l,k)$ takes each such matrix to the $l$-dimensional subspace of $\bbC^k$ spanned by the columns of the matrix.  
Thus elements of $G(l,k)$ are $k\times l$ matrices of rank $l$, modulo the action of $GL(l,\bbC)$. 

The rank condition can be expressed in terms of $l\times l$ minors $d_I$, where
\[
I=\{i_1,\dots,i_l\},\qquad 1\leq i_1<i_2<\dots<i_l\leq k,
\]
and $d_I$ is the determinant of the $l\times l$ matrix obtained by choosing the rows according to $I$. The condition is of course that $d_I\neq 0$ for 
at least one $I$.  The numbers $d_I$ are not completely determined by  the point $x$; namely, under the action of $g\in GL(l)$, they are all multiplied by the same scalar $\det g$. Thus the coordinates $(d_I)$, called the Pl\"ucker coordinates, define a point in the projective space $\bbP(\bigwedge^l(\bbC^k))$. In this way $G(l,k)$ is embedded into 
$\bbP(\bigwedge^l(\bbC^k))$, and the image is defined by certain equations called Pl\"ucker relations. We are not writing these equations down, because we will not need them explicitly.

If some $d_I\neq 0$ for the point $x$, we can assume, by renumbering the coordinates, that $I=\{1,\dots,l\}$. Then we can use the $GL(l)$-action to find a unique representative of $x$ of the form
\[
x=\begin{pmatrix} I_l \cr X
           \end{pmatrix},
\]
where $X$ is a $(k-l)\times l$ matrix. Conversely, any $(k-l)\times l$ matrix defines a unique point in the open subset $V_I$ of $G(l,k)$ defined by the condition $d_I\neq 0$. In this way we have covered $G(l,k)$ by $\binom{k}{l}$ open sets, each of which is isomorphic to the affine space $\bbC^{l(k-l)}$. 

We now return to our double fibration (\ref{doublefibcell})
\[
\xymatrix{
& & & & & V \ar[rd]^\tau\ar[ld]_\eta & & & &  \\
& & & &   W   & &   U   & &  \\ 
}
\]

with $U$ the big cell in $G/P$. 

Rearranging the variables, we can assume that $U$ consists of the $k$-planes $L_2$ in 
$\bbC^n$  which correspond to $n\times k$ matrices with the top $k\times k$ minor regular, and thus we can parametrize $U$ by matrices
\[
\begin{pmatrix} I_ k \cr
                         X
           \end{pmatrix}
\]
with $X$ an arbitrary $(n-k)\times k$ matrix. The points of $V$ are the partial flags 
\[
0\subset L_1\subset L_2\subset \bbC^n
\]
with $L_2$ as above and with $L_1$ an $l$-plane contained in $L_2$. If $L_1$ is given in homogeneous coordinates by an $n\times l$ matrix $Y$, then the columns of $Y$ must be linear combinations of the columns of the matrix $\begin{pmatrix} I_ k \cr  X \end{pmatrix}$ corresponding to $L_2$. In other words, there is a $k\times l$ matrix $Z$ such that
\[
Y=\begin{pmatrix} I_ k \cr
                         X
           \end{pmatrix} Z =
\begin{pmatrix} Z \cr
                         XZ
           \end{pmatrix}.
\]
Since $Y$ is of rank $l$, it follows that the rank of $Z$ must also be $l$, and so $Z$ defines a point in the Grassmanian $G(l,k)$ of $l$-planes in $\bbC^k$.
Note also that $Z$ is clearly uniquely determined by the point $Y=\begin{pmatrix} Z \cr XZ  \end{pmatrix}$.

We now define a map $\pi:W\to G(l,k)$ by setting 
\[
\pi\begin{pmatrix} Z \cr XZ  \end{pmatrix}=Z,\qquad \begin{pmatrix} Z \cr XZ  \end{pmatrix}\in W.
\]
Then $\pi$ is surjective by the above considerations; namely, for any $Z$ of rank $l$, we can construct $Y=\begin{pmatrix} Z \cr XZ  \end{pmatrix}\in W$, coming from a point in the fiber $\tau^{-1}\begin{pmatrix} I_k \cr X  \end{pmatrix}$. 

Let now $V_I$ be any of the affine spaces covering $G(l,k)$ as above, i.e., $V_I$ is given by the inequality $d_I\neq 0$. We claim that  
\[
\pi^{-1}(V_I) = V_I\times F,
\]
where $F$ is an affine space. In particular, $\pi$ is a locally trivial fibration.

To see this, we first renumber the coordinates, so that $I$ becomes $\{1,\dots,l\}$. Then every point in $V_I$ has non-homogeneous coordinates $\begin{pmatrix} I_l \cr Z'  \end{pmatrix}$
for some $(k-l)\times l$ matrix $Z'$. 

The fiber over the point  $\begin{pmatrix} I_l \cr Z'  \end{pmatrix}$ consists of points 
\[
 \begin{pmatrix} I_l \cr Z' \cr X \begin{pmatrix} I_l \cr Z'  \end{pmatrix} \end{pmatrix},
\]
with $X$ ranging over $(n-k)\times k$ matrices. Writing $X= \begin{pmatrix} X_1 & X_2  \end{pmatrix}$ with $X_1$ of size $(n-k)\times l$ and $X_2$ of size 
$(n-k)\times (k-l)$, we see that this fiber consists of points 
\[
 \begin{pmatrix} I_l \cr Z' \cr X_1+X_2Z' \end{pmatrix}.
\]
It follows that every point in the fiber is obtained for $X= \begin{pmatrix}X_1 & 0 \end{pmatrix}$,
with $X_1$ uniquely determined, and thus the fiber is the affine space $F$ of 
$(n-k)\times l$ matrices. The same $X_1$ describes the fibers over all points
$\begin{pmatrix} I_l \cr Z'  \end{pmatrix}$, so we conclude that $\pi^{-1}(V_I)=V_I\times F$, as claimed. We have proved

\begin{lem}
\label{steincover}
There is a locally trivial fibration $\pi:W\to G(l,k)$, such that if $V_I\subset G(l,k)$ are the affine spaces $d_I\neq 0$ described above, then
\[
\pi^{-1}(V_I)=V_I\times F,
\]
with the fiber $F$ isomorphic to an affine space.
\end{lem}

We need some basic properties of Stein spaces, which for our purposes may be defined as holomorphic manifolds on which the cohomology of all coherent sheaves vanishes in all positive degrees. All of the following facts, and much more, can be found in \cite{GR}.

\begin{prop} 
\label{stein}
\begin{enumerate}
\item The affine space $\bbC^n$ is Stein.
\item If $X$ is a Stein space, and if $Y\subset X$ is the complement of the zero set of a holomorphic function $f$ on $X$, then $Y$ is Stein.
\item If $X$ is a Stein space, and if $Y\subset X$ is the complement of the set of common zeros of $k$ holomorphic functions on $X$, then the cohomology of any coherent holomorphic sheaf on $Y$ vanishes in degrees larger than $k-1$.
\end{enumerate}
\end{prop}

We would now like to combine Lemma \ref{steincover} and 
Proposition \ref{stein}(3) to conclude the vanishing claimed in 
Theorem \ref{vanishingcoho}. 
However the number of the affine sets $V_I\times F$ covering $W$ is too big.

To get around this difficulty, we totally order the sets $I$, first by degree and then lexicographically within a fixed degree. The degree of $I=\{i_1,\dots,i_l\}$ is defined as
\[
|I|=\sum_{j=1}^l i_j - \frac{l(l+1)}{2}.
\]
The degree of $I$ varies from 0 (for $I=\{1,\dots,l\}$) to $l(k-l)$ (for $I=\{k-l+1,\dots,k\}$). 
For any $I$ as above, we define  
\[
U_I=\bigcup_{J\leq I} V_J\times F.
\]
Since $U_{\{k-l+1,\dots,k\}}=W$ and
$|\{k-l+1,\dots,k\}|=l(k-l)$, the following lemma implies Theorem \ref{vanishingcoho}.

\begin{lem} 
The cohomology of any coherent holomorphic sheaf on $U_I$ 
vanishes in degrees above $|I|$.
\end{lem}
\begin{proof}
Following \cite{S}, we will use the Mayer-Vietoris exact sequence to prove the statement by induction.

It is clear that the statement is true for the smallest index
$I=\{1,\dots,l\}$. Namely, in that case $U_I=V_I\times F$ is an affine space, therefore it is Stein by Lemma \ref{stein}(1), and so the cohomology of all coherent holomorphic sheaves on $U_I$ vanishes in degrees above 0.

Assume now that the statement is true for the immediate predecessor $I'$ of $I$ and let us prove it for $I$. Denote $|I|=i\geq 1$; then $|I'|\leq i$. The Mayer-Vietoris exact sequence for a holomorphic sheaf $\caS$ on $U_I=U_{I'}\cup (V_I\times F)$ is
\begin{multline*}
\dots\to H^j(U_I,\caS)\to H^j(U_{I'},\caS)\oplus H^j(V_I\times F,\caS)\to 
H^j(U_{I'}\cap (V_I\times F),\caS)\to \\
\to H^{j+1}(U_I,\caS)\to H^{j+1}(U_{I'},\caS)\oplus H^{j+1}(V_I\times F,\caS)\to\dots
\end{multline*}
By the inductive assumption, if $j\geq i$, then
$H^{j+1}(U_{I'},\caS)\oplus H^{j+1}(V_I\times F,\caS)=0$. Therefore we will be done if we prove that $H^j(U_{I'}\cap (V_I\times F),\caS)=0$. This is true 
by Proposition \ref{stein}(3), because 
$V_I\times F$ is Stein, and by Lemma \ref{equations} below, $U_{I'}\cap (V_I\times F)$ is a subset of $V_I\times F$ with complement given by $|I|= i\leq j$ equations. 
\end{proof}

\begin{lem} 
\label{equations}

Let $w\in W$ and let $I=\{i_1,\dots,i_l\}$ be such that $d_I(\pi(w))\neq 0$. 

Assume that $d_J(\pi(w))=0$ for all $J$ of the form  
\[
J=\{i_1,\dots,i_{r-1},j_r,j_{r+1},\dots,j_l\}
\]
for some $r\in\{1,\dots,l\}$, with $i_{r-1}<j_r<i_r$ and 
\[
j_{r+1},\dots,j_l\in\{i_r,\dots,i_l\}.
\]
(By convention, $i_0=0$.)
 
Then $d_K(\pi(w))=0$ for all $K<I$. Furthermore, the number of $J$ as above is 
$|I|$.
\end{lem}

\begin{proof}
We first show that the number of possible $J$ is exactly $|I|$. For a fixed $r$, 
the number of choices for $J$ is $i_r-i_{r-1}-1$ (for $j_r$) times $l-r+1$ (for $j_{r+1},\dots,j_l$). The total number of choices is thus
\begin{multline*}
\sum_{r=1}^l(i_r-i_{r-1}-1)(l-r+1)= 
\sum_{r=1}^l i_r(l-r+1)-\sum_{r=1}^l i_{r-1}(l-r+1)-\sum_{r=1}^l(l-r+1)=\\
\sum_{r=1}^l i_r(l-r+1)-\sum_{s=0}^{l-1} i_s(l-s)-\sum_{s=1}^l s=
\sum_{r=1}^l i_r-\frac{l(l+1)}{2}=|I|.
\end{multline*}

The assumption $d_I(w)\neq 0$ means that the rows $w_{i_1},\dots,w_{i_l}$ of the
$k\times l$ matrix corresponding to $w$ are linearly independent. The assumption
$d_J(w)=0$ for all $J$ as above means the following. If $i_{t-1}<s<i_t$ for some
$t\in\{1,\dots,l\}$, then $d_J(w)=0$ for every 
$J=\{i_1,\dots,i_{t-1},s,i_{t+1}',\dots,i_l'\}$ such that $\{i_{t+1}',\dots,i_l'\}\subset \{i_t,\dots,i_l\}$ means that $w_s$ can be expressed as a linear combination of $w_{i_1},\dots,w_{i_{t-1}}$ and any $l-t$ vectors among $w_{i_t},\dots,w_{i_l}$. This implies that $w_s$ is a linear combination of $w_{i_1},\dots,w_{i_{t-1}}$. 

Let now
\[
K=(k_1,\dots,k_l)<I.
\]
Then there is some $u\in\{1,\dots,l\}$ such that $k_u<i_u$; let us take the smallest such $u$. By the above considerations, $w_{k_1},\dots,w_{k_u}$ are in the linear span of $w_{i_1},\dots,w_{i_{u-1}}$. This means $w_{k_1},\dots,w_{k_u}$ are linearly dependent, and therefore $d_K(w)=0$.
\end{proof}

\end{document}